\documentclass[a4paper,11pt,reqno]{article}

\usepackage[utf8]{inputenc}
\usepackage{lmodern}
\usepackage[english]{babel}
\usepackage{microtype}

\usepackage{amsmath,amssymb,amsfonts,amsthm,mathrsfs}
\usepackage{braket}
\usepackage{bm}
\usepackage{aliascnt}

\usepackage[hidelinks]{hyperref}
\usepackage[a4paper,hmargin=3.25cm,bottom=4cm,top=3.5cm,footskip=3\baselineskip]{geometry}

\usepackage{subcaption}
\usepackage{tikz}

\makeatletter
\def\newaliasedtheorem#1[#2]#3{
  \newaliascnt{#1@alt}{#2}
  \newtheorem{#1}[#1@alt]{#3}
  \expandafter\newcommand\csname #1@altname\endcsname{#3}
}
\makeatother

\numberwithin{equation}{section}

\theoremstyle{plain}

\newaliasedtheorem{proposition}[theorem]{Proposition}
\newaliasedtheorem{lemma}[theorem]{Lemma}
\newaliasedtheorem{corollary}[theorem]{Corollary}
\newaliasedtheorem{counterexample}[theorem]{Counterexample}

\theoremstyle{definition}
\newaliasedtheorem{definition}[theorem]{Definition}
\newaliasedtheorem{question}[theorem]{Question}

\theoremstyle{remark}
\newaliasedtheorem{remark}[theorem]{Remark}

\newcommand{\step}[1]{\par\textbf{\textit{Step #1.}} \ignorespaces}

\newcommand{\setZ}{\mathbb{Z}}

\newcommand{\setR}{\mathbb{R}}

\newcommand{\setT}{\mathbb{T}}

\newcommand{\T}{\mathcal{T}}

\newcommand{\abs}[1]{\left\lvert#1\right\rvert}
\newcommand{\eps}{\varepsilon}
\newcommand{\plchldr}{\,\cdot\,}

\DeclareMathOperator{\Inv}{Inv}
\DeclareMathOperator{\Meas}{\mathscr{M}}
\DeclareMathOperator{\Prob}{\mathscr{P}}
\DeclareMathOperator{\supp}{supp}
\DeclareMathOperator{\Id}{Id}

\title{Counterexamples in multimarginal optimal transport \\
with Coulomb cost and spherically symmetric data}
\author{Maria Colombo
\thanks{Scuola Normale Superiore, Pisa, \url{maria.colombo@sns.it}.} \and
Federico Stra
\thanks{Scuola Normale Superiore, Pisa, \url{federico.stra@sns.it}.}}
\date{\today}

\begin{document}

\maketitle

\begin{abstract} We disprove a conjecture in Density Functional Theory, relative to multimarginal optimal transport maps with Coulomb cost. We also provide examples of maps satisfying optimality conditions for special classes of data.
\end{abstract}


\section{Introduction}
A natural problem in Quantum Physics consists in studying the behavior of $N$ electrons subject to the interaction with some nuclei, their mutual interaction and the effect of an external potential. In this setting, a relevant quantity is the ground state energy of the system, which can be found by solving the Schr\"odinger equation. However, this procedure is computationally very costly even for a small number of electrons; Density Functional Theory proposes an alternative method to compute the ground state energy and was first introduced by Hohenberg and Kohn \cite{HK} and then by Kohn and Sham \cite{KS}.

In \cite{BDG,CF} the authors present a mathematical model for the strong interaction limit of Density Functional Theory; they study the minimal interaction of $N$ electrons and the semiclassical limit of DFT. The model is based on Monge multimarginal optimal transport (see also the recent survey \cite{DGN}, where the state of the art on this problem is described), which consists in the minimization problem
\begin{equation}\label{monge-old}
(M) = \inf\Set{ \int_{\setR^n}  C\bigl(x,T_2(x), \dotsc, T_N(x)\bigr)  \, d\rho(x)
	: T_2,\dotsc, T_N \in \T(\rho) },
\end{equation}
where $\rho\in \Prob(\setR^n)$ is a given probability measure,
$C:(\setR^n)^N\to [0,\infty]$ is the Coulomb interaction
\begin{equation}\label{defn:coulomb-cost}
C(x_1, \dotsc, x_N) = \sum_{1 \leq i<j \leq N} \frac{1}{|x_i-x_j|}
\qquad \forall (x_1, \dots, x_N) \in (\setR^n)^N,
\end{equation}
and $\T(\rho)$ is the set of admissible transport maps
\[
\T(\rho)=\set{ T:\setR^n \to \setR^n \text{ Borel}: T_\sharp \rho=\rho }.
\]
Since the cost is symmetric, a natural variant of the Monge problem allows only
cyclical maps
\[
(M_\text{cycl}) = \inf\Set{ \int_{\setR^n} C\bigl(x,T(x), \dotsc, T^{(N-1)}(x)\bigr)
	\, d\rho(x) : T\in \T(\rho),\, T^{(N)}=\Id }
\]
where with $T^{(k)}$ we denote the composition of $T$ with itself for $k$ times.
Following the standard theory of optimal transport
(see \cite{Villani09,Ambrosiolectures}), we also introduce the Kantorovich problem
\[
(K) = \min \Set{ \int _{(\setR^n)^N} c(x_1, \dotsc, x_N) \, d\gamma (x_1, \dotsc, x_N)
	: \gamma \in \Pi(\rho) },
\]
where $\Pi(\rho)$ is the set of transport plans
\[
\Pi(\rho)=\Set{\gamma\in\Prob(\setR^{nN})
	:\pi^i_\sharp \gamma= \rho, \, i=1, \dotsc, N }
\]
and $\pi^i:(\setR^n)^N \to \setR^n$ are the projections on the $i$-th component for
$i=1,\dotsc,N$. To every $(N-1)$-uple of transport maps
$T_2,\dotsc,T_N \in \T(\rho)$ we canonically associate the transport plan
$\gamma = (Id, T_2,\dotsc, T_N)_\sharp \rho \in \Pi (\rho)$. As proved in \cite{CD},
if $\rho$ is non-atomic the values of the minimum problems coincide
\[
(K)= (M)= (M_\text{cycl}).
\]

Existence of optimal transport plans in $(K)$ follows from a standard compactness and lower semicontinuity argument. In turn, existence of optimal maps in $(M)$ is largely open; it is understood only with $N=2$ marginals in any dimension $n$ and in dimension $n=1$ with any number $N$ of marginals (see \cite{CF} and \cite{CDD} respectively). In a different context, optimal cyclical maps as in $(M_\text{cycl})$ appear in \cite{GMo} for some particular costs generated by vector fields.

As regards uniqueness of optimal symmetric plans with Coulomb cost, it holds in dimension $1$, but, as shown in \cite{P2}, it fails in the same class already when we consider spherically symmetric densities in $\setR^2$, for any $N$. On the other hand,
the Kantorovich duality works also for this cost (see \cite{RR}) and the dual problem admits maximizers (namely, Kantorovich's potentials), as shown by De Pascale \cite{DeP}; moreover, in \cite{CFP} the limit of symmetric optimal plans as $N\to\infty$ is shown to be the infinite product measure of $\rho$ with itself.

Beyond the $1$-dimensional case, which is well understood, a physically relevant case is given by spherically symmetric densities $\rho$ in $\setR^n$, with any number of marginals. In the physics literature, they appear in \cite{S99,SGS} to study simple atoms like Helium ($N=2$), Litium ($N=3$), and Berillium ($N=4$). In this case the problem reduces, thanks to the spherical symmetry, to a problem in $1$-dimension, with a more complicated cost function (see \cite{P2}, where this reduction is rigorously described). In the class of admissible transport maps for problem $(M_\text{cycl})$, Seidl, Gori Giorgi and Savin identified some particularly simple maps: roughly speaking, they divide $\setR^n$ in $N$ spherical shells, each containing one electron in average, and consider the transport maps which send each shell onto the next one by a monotonically increasing or decreasing map. They conjecture the optimality of one of these maps in $(M_\text{cycl})$.

In the following, we provide counterexamples to the conjecture showing that there are cases in which none of these maps is optimal in problem $(M_\text{cycl})$. On the other hand, we also point out situations where some of these maps satisfy optimality conditions, namely $c$-monotonicity. We deal for simplicity with radial measures in $\setR^2$ with $3$ marginals, although similar examples and computations can be carried out in any dimension and with any number of marginals.

The plan of the paper is the following. In Section~\ref{sec:e-c} we present the problem with spherically symmetric data, we recall the notion of $c$-monotonicity and a few properties of optimal transport maps, and we give some examples and counterexamples.
In Sections~\ref{sec:taylor} and~\ref{sec:spread} we study the properties of the cost for close radii and for spread apart radii, respectively.
In Section~\ref{sec:proofs} we apply these properties to give rigorous proofs of the examples and counterexamples.

\section{Examples and counterexamples}\label{sec:e-c}

\subsection{Monge and Kantorovich problems with radial densities}
As we mentioned above, the transport problem~\eqref{monge-old} reduces to a
$1$-dimensional one (i.e., by proving that spheres get mapped to spheres),
as rigorously done in \cite{P2}.
\begin{figure}
\centering
\begin{tikzpicture}[scale=.8]
\draw[dashed] (0,0) circle[radius=1.5] circle[radius=2.5] circle[radius=3.5];
\fill (0,0) circle [radius=.05];
\fill[radius=.08] (0:1.5) circle node[right] {$v_1$}
                  (110:2.5) circle node[above left] {$v_2$}
                  (220:3.5) circle node[below left] {$v_3$};
\draw (0,0) -- node[below] {$r_1$} (0:1.5)
      (0,0) -- node[pos=.75, left] {$r_2$} (110:2.5)
      (0,0) -- node[pos=.9, below right] {$r_3$} (220:3.5);
\draw[->] (.4,0) arc (0:110:.4);
\draw[->] (.25,0) arc (0:220:.25);
\draw +(55:.4) node[above] {$\theta_2$}
      +(180:.25) node[left] {$\theta_3$};
\end{tikzpicture}
\caption{A configuration of three charges at distances $r_1$, $r_2$ and $r_3$
with angles $\theta_2$ and $\theta_3$.}
\label{fig:three-charges}
\end{figure}
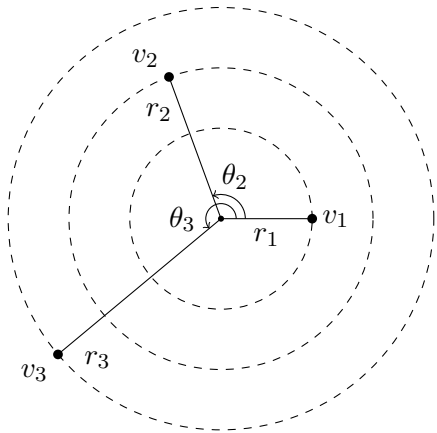
Assuming from now on $N=3$, given three radii $r_1,r_2,r_3 \in \setR_+=(0,\infty)$,
we consider the associated \emph{exact cost} (see \autoref{fig:three-charges})
\begin{equation}\label{eqn:c-cost}
c(r_1, r_2, r_3)
= \min \Set{ \frac1{\abs{v_2-v_1}} + \frac1{\abs{v_3-v_2}} + \frac1{\abs{v_1-v_3}}
 : \abs{v_i}=r_i,\ i=1,2,3},
\end{equation}
which is a positive, symmetric, continuous function. Let us denote $(0,\infty)$ by $\setR_+$. Given a non-atomic probability measure $\rho \in \Prob (\setR_+)$,
the set of transport maps reads as
\[
\T(\rho) = \set{ T:\setR_+ \to \setR_+ \text{ Borel}: T_\sharp \rho=\rho },
\]
and the cyclical Monge problem corresponding to~\eqref{monge-old} can be written as
\begin{equation}\label{monge}
(M_\text{cycl}) = \inf\Set{ \int_{\setR_+} c\bigl(x,T(x), T^{(2)}(x)\bigr) \, d\rho(x)
	: T\in \T(\rho),\, T^{(3)}=\Id }.
\end{equation}
We also introduce the set of transport plans
\[
\Pi(\rho)=\set{\gamma \in \Prob (\setR_+^{3})
	:\pi^i_\sharp \gamma= \rho, \, i=1,2,3 },
\]
where $\pi^i:(\setR_+)^3 \to \setR_+$ are the projections on the $i$-th component
for $i=1,\dotsc,3$, and the Kantorovich multimarginal problem
\begin{equation}\label{kantorovich}
(K)=\min \Set{ \int _{(\setR_+)^3} c(r_1, r_2, r_3) \,d\gamma (r_1, r_2, r_3)
	: \gamma \in \Pi(\rho) }.
\end{equation}

\subsection{Some special maps}
In the following definition, we introduce some special transport maps, which were conjectured in~\cite{SGS} to be good candidates for optimality in problem~\eqref{monge}.

\begin{definition}
Let $\rho \in \Meas(\setR_+)$ be a non-atomic probability measure and let
$d_1, d_2 \in \setR_+$ such that
$\rho([0,d_1])= \rho([d_1,d_2])= \rho([d_2, \infty])= 1/3$.
The $DDI$-map $T: \setR_+ \to\setR_+$ associated to $\rho$ is the unique
(up to $\rho$-negligible sets) map such that $T_\sharp \rho = \rho$ and
\begin{itemize}
\item $T$ maps $(0,d_1)$ onto $(d_1,d_2)$ decreasingly,
\item $T$ maps $(d_1,d_2)$ onto $(d_2,\infty)$ decreasingly,
\item $T$ maps $(d_2,\infty)$ onto $(0,d_1)$ increasingly.
\end{itemize}
Similarly, we define, for instance, the $DID$-map mapping
$(0,d_1)$ onto $(d_1,d_2)$ decreasingly,
$(d_1,d_2)$ onto $(d_2,\infty)$ increasingly and
$(d_2,\infty)$ onto $(0,d_1)$ decreasingly.

The $\{D,I\}^3$-class associated to $\rho$ is composed by the maps with all
the possible monotonicities, under the condition that $T^{(3)}=\Id$:
therefore we have $III$, $IDD$, $DID$ and $DDI$, (see \autoref{fig:DI3-class}).
\end{definition}

\begin{figure}
\centering
\begin{subfigure}{.25\textwidth}
\centering
\begin{tikzpicture}[scale=.77]
\draw[<->] (0,3.5) node[above] {$T(r)$} |- (3.5,0) node[below=4pt] {$r$};
\foreach \t in {0,1,2,3}
	\draw (\t,1pt) -- (\t,-3pt) node[below] {$\t$}
		(1pt,\t) -- (-3pt,\t) node[left] {$\t$};
\draw (0,1) -- (2,3) (2,0) -- (3,1);
\end{tikzpicture}
\caption{$III$ map.}
\end{subfigure}\begin{subfigure}{.25\textwidth}
\centering
\begin{tikzpicture}[scale=.77]
\draw[<->] (0,3.5) node[above] {$T(r)$} |- (3.5,0) node[below=4pt] {$r$};
\foreach \t in {0,1,2,3}
	\draw (\t,1pt) -- (\t,-3pt) node[below] {$\t$}
		(1pt,\t) -- (-3pt,\t) node[left] {$\t$};
\draw (0,1) -- (1,2) (1,3) -- (2,2) (2,1) -- (3,0);
\end{tikzpicture}
\caption{$IDD$ map.}
\end{subfigure}\begin{subfigure}{.25\textwidth}
\centering
\begin{tikzpicture}[scale=.77]
\draw[<->] (0,3.5) node[above] {$T(r)$} |- (3.5,0) node[below=4pt] {$r$};
\foreach \t in {0,1,2,3}
	\draw (\t,1pt) -- (\t,-3pt) node[below] {$\t$}
		(1pt,\t) -- (-3pt,\t) node[left] {$\t$};
\draw (0,2) -- (1,1) (1,2) -- (2,3) (2,1) -- (3,0);
\end{tikzpicture}
\caption{$DID$ map.}
\end{subfigure}\begin{subfigure}{.25\textwidth}
\centering
\begin{tikzpicture}[scale=.77]
\draw[<->] (0,3.5) node[above] {$T(r)$} |- (3.5,0) node[below=4pt] {$r$};
\foreach \t in {0,1,2,3}
	\draw (\t,1pt) -- (\t,-3pt) node[below] {$\t$}
		(1pt,\t) -- (-3pt,\t) node[left] {$\t$};
\draw (0,2) -- (1,1) (1,3) -- (2,2) (2,0) -- (3,1);
\end{tikzpicture}
\caption{$DDI$-map.}
\end{subfigure}
\caption{The four types of maps considered in the conjecture in the case of a
uniform density on $[0,3]$.}\label{fig:DI3-class}
\end{figure}
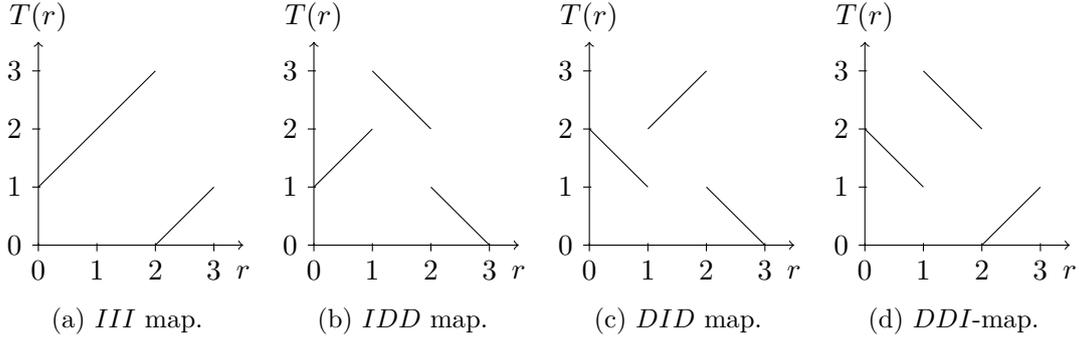

In the rest of the paper we answer the following question:
\begin{question}
Is the $DDI$-map associated to $\rho$ optimal in problem \eqref{monge} for every
measure $\rho \in \Prob (\setR_+)$?
Is one of the maps in $\{D,I\}^3$-class associated to $\rho$ optimal in
problem~\eqref{monge} for every non-atomic probability measure
$\rho \in \Prob (\setR_+)$?
\end{question}

\subsection{\texorpdfstring{A necessary condition for optimality: $\bm{c}$}{c}-monotonicity}
Before presenting the examples and counterexamples,
we recall a well-known optimality condition in optimal transport.
\begin{definition}\label{defn:c-mon}
Let $c:(\setR_+)^N \to [0, \infty]$ be a cost function. We say that a set
$\Gamma \subset (\setR_+)^N$ is $c$-monotone with respect to $p \subseteq \{1,\dotsc,N\}$ if
\begin{equation}\label{eqn:c-mon-con-p}
 c(x)+c(y) \leq c({X(x,y,p)})+ c({Y(x,y,p)}) \qquad \forall x,y \in \Gamma,
\end{equation}
where $X(x,y,p), Y(x,y,p) \in (\setR_+)^N$ are obtained from $x$ and $y$ by exchanging their coordinates on the complement of $p$, namely
\begin{equation}
X_i(x,y,p)= \left\{
\begin{array}{ll}
x_i &\mbox{if } i\in p \\
y_i&\mbox{if } i\notin p \\
\end{array}
\right.
\qquad
Y_i(x,y,p)= \left\{
\begin{array}{ll}
y_i &\mbox{if } i\in p \\
x_i&\mbox{if } i\notin p \\
\end{array}
\right.
\qquad \forall i\in \{1,...,N\}.
\end{equation}
We say that $\Gamma \subset (\setR_+)^N$ is  $c$-monotone if \eqref{eqn:c-mon-con-p} holds
true for every $p \subseteq \{1,\dotsc,N\}$.
\end{definition}

Let $\gamma \in \Pi(\rho)$ be a transport plan.
The following Proposition (\cite[Lemma 2]{pass2012local}, see also \cite[Proposition 2.2]{CDD}, where the result is used to describe optimal maps with Coulomb cost in $1$ dimension) presents a necessary condition for optimality of $\gamma$.

\begin{proposition}\label{monotonia}
Let $c:(\setR_+)^3 \to [0, \infty]$ be a continuous cost and let $\rho$ be a probability measure on $(\setR_+)$.
Let $\gamma \in \Pi(\rho)$ be an optimal transport plan for problem \eqref{kantorovich} and assume $(K)<\infty$ (therefore $\gamma$ has finite cost).
Then $\supp \gamma$ is $c$-monotone.
\end{proposition}

\begin{remark}
Given an optimal plan $\gamma$, the support of $\gamma$ is $c$-monotone even in a stronger sense than the one in \autoref{defn:c-mon}. More precisely, given two points $x$ and $y$ (for simplicity, assume that all their coordinates are distinct to avoid multiplicity issues), we have that
\begin{equation}
\label{eqn:stronger-c-monot}
c(x)+c(y) \leq c(X)+ c(Y)
\end{equation}
for every choice of $X,Y \in (\setR_+)^N$ such that the union of the coordinates of $X$ and $Y$ is the same as the union of the coordinates of $x$ and $y$. Indeed, given any permutation $\sigma$ of the coordinates of $(\setR_+)^N$, we have that $\sigma (y)$ is in the support of the symmetrization of $\gamma$, which is still optimal because of the symmetry of the optimal plan. Hence, applying \autoref{monotonia} to $x$ and $\sigma(y)$, we obtain \eqref{eqn:stronger-c-monot} for any $X$ and $Y$.
\end{remark}

\subsection{Counterexamples}
The first example shows that the $DDI$-map is not always optimal in problem~\eqref{monge}, by taking as marginal a measure which is concentrated in a small neighborhood of the unit sphere. 

\begin{counterexample}\label{ce:145}
There exists $\eps>0$ such that, setting 
\[
\rho_\eps = \frac{1}{12 \eps} 1_{[1,1+12\eps]} \, dr \in \Meas(\setR_+),
\]
the $DDI$-map associated to $\rho_\eps$ is not $c$-monotone and, therefore,
not optimal in problem~\eqref{monge}.
\end{counterexample}

The proof is based on the analysis of $c$-monotonicity for similar radii, obtained by Taylor expanding the cost around the point $(1,1,1)$. The analysis of $c$-monotone sets in this context suggests that the $DID$-map may be optimal in this example.

The next example modifies the previous one by sending $1/6$ of the total mass far away; in this way, the cost of the orbits of these points (which have two coordinates close to $1$ and one large coordinate) can be easily computed. Thanks to this property, we can show that none of the maps in the $\{ D, I\}^3$-class can be optimal, since their support is not $c$-monotone.

\begin{counterexample}\label{ce:class}
There exist $M,\eps>0$ such that, setting 
\[
\rho_{M,\eps} =\Big( \frac{1}{6 \eps} 1_{[1,1+5\eps]} +
\frac{1}{6} 1_{[M, M+1]} \Big)\, dr \in \Meas(\setR_+),
\]
none of the maps in the $\{ D, I\}^3$-class associated to $\rho_{M,\eps}$ is  optimal
in problem~\eqref{monge}.
\end{counterexample}

\begin{remark}
In \autoref{ce:4marg} we will see a similar result for the problem with $4$ marginals. However, we preferred to restrict the presentation to the case with $3$ marginals since the ideas involved are the same, but the computations are easier.
\end{remark}

There are particular measures $\rho$ for which the $DDI$-map is $c$-monotone (whereas this property fails in \autoref{ce:145} and \ref{ce:class}). For this reason one may expect that this map is also optimal in problem~\eqref{monge}, but, to show this, sufficient conditions for optimality (stronger than $c$-monotonicity) would have to be identified.

\begin{proposition}[Examples of $c$-monotone $DDI$-maps]\label{prop:example}
There exists $M>0$ such that for any probability measure $\rho$ such that
$\rho([1,2])= \rho([3,4])= \rho([M,\infty))=1/3$
the $DDI$-map is $c$-monotone (according to \autoref{defn:c-mon}).
\end{proposition}

\begin{figure}
\centering
\begin{subfigure}{.45\textwidth}
\centering
\begin{tikzpicture}[scale=.8]
\draw[dashed] (0,0) circle[radius=3.5];
\fill (0,0) circle [radius=.05];
\fill[radius=.08] (0:3.5) circle node[right] {$v_1$}
                  (120:3.5) circle node[above left] {$v_2$}
                  (-120:3.5) circle node[below left] {$v_3$};
\draw (0,0) -- node[pos=.7, above] {$r_1$} (0:3.5)
      (0,0) -- node[pos=.5, below left] {$r_1$} (120:3.5)
      (0,0) -- node[pos=.5, above left] {$r_1$} (-120:3.5);
\draw[->] (.3,0) arc (0:120:.3);
\draw[->] (.3,0) arc (0:-120:.3);
\draw +(90:.4) node[above right] {$2/3\pi$}
      +(-90:.2) node[below right] {$-2/3\pi$};
\end{tikzpicture}
\caption{A configuration of three charges at the same distance $r_1$ from the
origin with angles $\theta_2=2/3\pi$ and $\theta_3=-2/3\pi$.}
\label{fig:three-charges-same-radius}
\end{subfigure}
\quad\,
\begin{subfigure}{.45\textwidth}
\centering
\begin{tikzpicture}[scale=.8]
\draw[dashed] (0,0) circle[radius=1.5] circle[radius=2.5] circle[radius=3.5];
\fill (0,0) circle [radius=.05];
\fill[radius=.08] (0:1.5) circle node[above right] {$v_1$}
                  (180:2.5) circle node[left] {$v_2$}
                  (0:3.5) circle node[right] {$v_3$};
\draw (0,0) -- node[pos=.7, below] {$r_1$} (0:1.5)
      (0,0) -- node[pos=.85, below] {$r_2$} (180:2.5)
      (0,0) -- node[pos=.85, below] {$r_3$} (0:3.5);
\draw[->] (.3,0) arc (0:180:.3);
\draw +(90:.3) node[above] {$\pi$};
\end{tikzpicture}
\caption{A configuration of three charges at distances $r_1$, $r_2$ and $r_3$
with angles $\theta_2=\pi$ and $\theta_3=0$.}
\label{fig:three-charges-spread-apart}
\end{subfigure}
\end{figure}

\section{Taylor expansion of the cost at \texorpdfstring{$\bm{r_1=r_2=r_3=1}$}{r1=r2=r3=1}}\label{sec:taylor}
 
In this section we want to address the following problem: given three radii
$r_1(t)$, $r_2(t)$ and $r_3(t)$ parametrized by $t\in\setR$ and starting from
the value $1$ at $t=0$, what is the expansion of
$c\bigl(r_1(t),r_2(t),r_2(t)\bigr)$ in powers of $t$ at $t=0$?

First, we notice that at $t=0$ the optimal angles are $\pm 2/3\pi$ and $c(1,1,1)=\sqrt 3$.
Indeed, given three unitary vectors $v_1,v_2,v_3$, calling $\alpha_1,\alpha_2, \alpha_3$ the angles of the triangle with vertices $v_1,v_2,v_3$ we have that $|v_1-v_2| = 2 \sin \alpha_3$ (and cyclical)
and therefore, by Jensen's inequality and by the convexity of $\alpha \mapsto (\sin \alpha)^{-1}$ in $[0,\pi]$,
\begin{equation}
\label{eqn:2pi3-ottimo}
\frac{1}{\abs{v_2-v_1}} + \frac{1}{\abs{v_3-v_2}} + \frac{1}{\abs{v_1-v_3}} =
\frac{1}{2} \sum_{i=1}^3 \frac{1}{\sin \alpha_i} \geq
\frac{3}{2\sin ((\alpha_1+\alpha_2+\alpha_3)/3)} = \sqrt 3, 
\end{equation}
with equality if and only if the triangle is equilateral.

Taking the angles to be exactly $\pm 2/3\pi$ leads to the following cost
\begin{equation}
\label{eqn:c-triangle}
c_\triangle(r_1,r_2,r_3) :=
\frac{1}{\sqrt{r_1^2 + r_1 r_2 + r_2^2}} +
\frac{1}{\sqrt{r_2^2 + r_2 r_3 + r_3^2}} +
\frac{1}{\sqrt{r_1^2 + r_1 r_3 + r_3^2}}\geq c(r_1,r_2,r_3).
\end{equation}
However the inequality is strict as soon as the three radii are different and the approximation of $c$ with $c_\triangle$ is too rough to deduce that they enjoy the same $c$-monotonicity structures. Therefore, we perform a finer analysis.

We want to take into account only the first order variation
of the radii as functions of $t$, so it is natural to consider three linearly
varying radii
\[
r_1(t) = 1+a_1t, \qquad
r_2(t) = 1+a_2t, \qquad
r_3(t) = 1+a_3t
\]
where $a_1,a_2,a_3\in\setR$ are some constants.
To these radii we associate the exact cost
\begin{equation}
\label{defn:g}
g(a_1,a_2,a_3,t) = c(1+a_1t,1+a_2t,1+a_3t),
\end{equation}
and we study the expansion of this function near $t=0$.

\begin{lemma}\label{lemma:derivatives} Let $a_1,a_2,a_3\in\setR$ and let $g$ be as in \eqref{defn:g}. Then we have that
\[
g(a,b,c,0) = \sqrt3.
\]
\[
\frac{\partial g}{\partial t}(a_1,a_2,a_3,0) = -\frac{a_1+a_2+a_3}{\sqrt3},
\]
\[
\frac{\partial^2 g}{\partial t^2}(a_1,a_2,a_3,0) =
\frac{4(a_1^2+a_2^2+a_3^2) + 6(a_1a_2+a_2a_3+a_3a_1)}{5\sqrt3},
\]
\begin{equation}
\begin{split}
\frac{\partial^3 g}{\partial t^3}(a_1,a_2,a_3,0) =
&\frac{308(a_1^3+a_2^3+a_3^3)}{375\sqrt3} \\
&+\frac{888(a_1^2a_2+a_1a_2^2+a_2^2a_3+a_2a_3^2+a_3^2a_1+a_3a_1^2) + 498a_1a_2a_3}
{375\sqrt3}.
\end{split}
\end{equation}
\end{lemma}

In the proof, we will write the Coulomb potential of three charges in terms of the distances from the origin and the angles between the charges. Given three radii $r_1$, $r_2$, $r_3$ and two angles $\theta_2$ and $\theta_3$,
we define the \emph{Coulomb potential} of the configuration of charges
depicted in \autoref{fig:three-charges}: 
\begin{equation}\label{eq:coulomb}
C(r_1, r_2, r_3, \theta_2, \theta_3) =
\frac{1}{\abs{v_2-v_1}} + \frac{1}{\abs{v_3-v_2}} + \frac{1}{\abs{v_1-v_3}}
\end{equation}
where
\[
v_1 = (r_1, 0), \qquad
v_2 = r_2 (\cos\theta_2, \sin\theta_2), \qquad
v_3 = r_3 (\cos\theta_3, \sin\theta_3).
\]
By definition of $c$, we notice that
\begin{equation}\label{eq:exact-cost}
c(r_1, r_2, r_3) = \min_{\theta_2, \theta_3 \in\setR}
C(r_1, r_2, r_3, \theta_2, \theta_3).
\end{equation}

\begin{proof}[Proof of \autoref{lemma:derivatives}]
For $t\in\setR$ and $\theta=(\theta_2,\theta_3)\in\setR^2$ we define also
the function
\[
G(t,\theta) = C(1+a_1t, 1+a_2t, 1+a_3t, \theta_2, \theta_3).
\]
Then $g(t) = G\bigl(t,\theta_0(t)\bigr)$ where $\theta_0(t)$ is the pair
of angles which minimizes \eqref{eq:exact-cost}.
From this optimality condition we know that
\[
G_\theta\bigl(t, \theta_0(t)\bigr) = 0.
\]
We want to apply the implicit function theorem to find the behavior of
$\theta_0(t)$. It's easy to check that $\theta_0(0) = (2/3\pi, -2/3\pi)$ and a direct computation shows that
\[
G_{\theta\theta}\bigl(0,\theta_0(0)\bigr) =
\frac{5}{6\sqrt{3}} \begin{pmatrix} 1 & -1/2 \\ -1/2 & 1 \end{pmatrix}\in \Inv(\setR^2;\setR^2).
\]
Therefore $\theta_0\in C^\infty\bigl((-\eps,\eps)\bigr)$ for some $\eps>0$ and we can compute its
derivatives in $0$. In particular, we have that
\begin{equation}\label{eq:opt-angles}
\theta_0'(0) = G_{\theta\theta}^{-1} \cdot G_{t\theta}
	\Bigr\rvert_{(0,\theta_0(0))} =
\frac{1}{5\sqrt{3}}\begin{pmatrix} -a_1-a_2+2a_3 \\ a_1-2a_2+a_3 \end{pmatrix}.
\end{equation}
The idea is now to consider the first order approximation
\[
\bar\theta(t) = \theta_0(0) + \theta_0'(0)t =
\begin{pmatrix} 2/3\pi \\ -2/3\pi \end{pmatrix} +
\frac{1}{5\sqrt{3}}\begin{pmatrix} -a_1-a_2+2a_3 \\ a_1-2a_2+a_3\end{pmatrix} t
\]
and the perturbed cost
\[
h(t) = G\bigl(t,\bar\theta(t)\bigr).
\]
We claim that $h(t) = g(t) + o(t^3)$, namely
\[
h(0) = g(0), \qquad h'(0) = g'(0), \qquad h''(0) = g''(0), \qquad
h'''(0) = g'''(0).
\]
The first two are clearly true, since $\bar\theta(0)=\theta_0(0)$ and
$\bar\theta'(0) = \theta_0'(0)$ by definition.
Now consider the function $t\mapsto G\bigl(t, \theta(t)\bigr)$, where
$\theta$ is either $\theta_0$ or $\bar\theta$.
To prove the claim, we show that its second and third
derivatives at $t=0$ depend only on $\theta'(0)$ and not on the second and third
derivatives of $\theta$.

As a matter of fact, we have
\[
\frac{d^2 G\bigl(t, \theta(t)\bigr)}{dt^2}\Biggr\rvert_{t=0} =
G_{tt} + 2G_{t\theta}\theta' + G_{\theta\theta}\theta'\theta' +
	G_\theta\theta'' \Bigr\rvert_{t=0},
\]
but $G_\theta\bigl(0,\theta(0)\bigr)=0$, so the second derivative does not
depend on $\theta''(0)$.
In a similar fashion, we have
\[
\frac{d^3 G\bigl(t, \theta(t)\bigr)}{dt^3}\Biggr\rvert_{t=0} =
G_{ttt}+ 3G_{tt\theta} \theta' + 3G_{t\theta\theta} (\theta')^2+ G_{\theta\theta\theta} (\theta')^3 +
3\left(G_{t\theta}+G_{\theta\theta}\theta'\right)\theta'' +
G_\theta\theta''' \Bigr\rvert_{t=0}.
\]
Again, $G_\theta\bigl(0,\theta(0)\bigr) = 0$, therefore $\theta'''(0)$ doesn't
contribute. Furthermore, we have $G_\theta\bigl(t,\theta_0(t)\bigr) = 0$, so that
differentiating in $t$ yields
\[
G_{t\theta}\bigl(0,\theta_0(0)\bigr) +
G_{\theta\theta}\bigl(0,\theta_0(0)\bigr)\theta_0'(0) = 0.
\]
But then also
\[
G_{t\theta}\bigl(0,\bar\theta(0)\bigr) +
G_{\theta\theta}\bigl(0,\bar\theta(0)\bigr)\bar\theta'(0) = 0,
\]
since $\bar\theta'(0)=\theta_0'(0)$.
Therefore we see that in both cases the coefficient of $\theta''$ vanishes.
This concludes the proof of the claim because we have shown that the
first three derivatives of $h$ and $g$ coincide at $t=0$.

At this point the derivatives of $h$ can be computed directly, since $h(a_1,a_2,a_3, \plchldr)$ is an explicit function of the last variable.
\end{proof}

In \autoref{lemma:derivatives} we found the first nontrivial Taylor term in the expansion of $g(t)$. We employ this computation to obtain informations on the $c$-monotonicity of points with linearly spaced radii close to $t=0$.
\begin{lemma}\label{lemma:120asint}
For every $t>0$, consider six linearly spaced radii
\begin{equation}
\label{eqn:r1-6}
(r_1, r_2, r_3, r_4, r_5, r_6) = (1, 1+t, 1+2t, 1+3t, 1+4t, 1+5t).
\end{equation}
Then there exists $t_0>0$ such that, for every $t\leq t_0$,
\[
c(r_1, r_4, r_6) + c(r_2, r_3, r_5) < c(r_1, r_4, r_5) + c(r_2, r_3, r_6).
\]
\end{lemma}

\begin{proof}
Let us define \[
F(t) = g(0,3,5,t) + g(1,2,4,t) - g(0,3,4,t) - g(1,2,5,t)
\]
Applying \autoref{lemma:derivatives} we can compute the derivatives of $F$ and find that
\[
F(0) = 0, \qquad F'(0) = 0, \qquad F''(0) = 0, \qquad F'''(0) = -\frac{284 \sqrt 3 }{125} < 0;
\]
this shows that $F(t)<0$ for $t$ sufficiently small and proves the lemma.
\end{proof}

\begin{remark}
Considering $r_1,...,r_6$ as in \eqref{eqn:r1-6}, one could prove that the choice $146$-$235$ is optimal between all possible choices, namely
\begin{multline}
\label{eqn:146otpimal}
c(r_1, r_4, r_6) + c(r_2, r_3, r_5) \\
= \min\Set{ c(p_1,p_2,p_3)+ c(p_4,p_5,p_6) : \{p_1,\dotsc,p_6\} = \{r_1,\dotsc,r_6\} },
\end{multline}
for $t$ small enough. Moreover, one could see that \eqref{eqn:146otpimal} holds also if we replace $c$ with $c_\triangle$ defined in \eqref{eqn:c-triangle}. This is, however, not needed for our counterexamples.
\end{remark}

\begin{remark}[Asymptotic expansion of the cost at infinity]
Although they will not be used in the proofs of the main results, we report the following formulas since they might help in future studies to gain more insight into the structure of $c$-monotone sets. We are interested in the asymptotic expansion of the cost as some of the radii go to infinity and the others remain fixed.

For $(r_1,r_2,r_3) = (1, 1, r)$, the optimal angles are
\[
\theta_2(r) = \pi - \frac{8}{r^2} + o\left(\frac{1}{r^3}\right), \qquad
\theta_3(r) = -\frac{\pi}{2} - \frac{4}{r^2} + o\left(\frac{1}{r^3}\right).
\]
In comparison to \eqref{eq:opt-angles}, this expansion is harder to justify (but can be easily verified numerically). However, from this fact it follows rigorously that the cost has the following asymptotic behaviour:
\[
\begin{split}
c(1, 1, r) &= C(1, 1, r, \pi, -\pi/2) - \frac{4}{r^4} +
	o\left(\frac{1}{r^4}\right) \\
&= \left(\frac12 + \frac{1}{\sqrt{1+r^2}}\right) - \frac{4}{r^4} +
	o\left(\frac{1}{r^4}\right).
\end{split}
\]
Similarly, for $(r_1,r_2,r_3) = (1, r, r)$, the optimal angles are
\[
\theta_2(r) = \frac{\pi}{2} + \frac{4}{r} + o\left(\frac{1}{r^2}\right), \qquad
\theta_3(r) = -\frac{\pi}{2} - \frac{4}{r} + o\left(\frac{1}{r^2}\right),
\]
and the cost is
\[
\begin{split}
c(1, r, r) &= C(1, r, r, \pi/2, -\pi/2) - \frac{4}{r^3} +
	o\left(\frac{1}{r^4}\right) \\
&= \frac{1}{2r} + \frac{2}{\sqrt{1+r^2}} - \frac{4}{r^3} +
	o\left(\frac{1}{r^4}\right).
\end{split}
\]
Furthermore, one can verify that
\[
c(1, r, r) = C\left(1, r, r, \frac\pi2+\frac4r, -\frac\pi2-\frac4r \right)
	- O\left(\frac1{r^7}\right).
\]
\end{remark}

\section{Condition for \texorpdfstring{$\bm{c=c_\pi}$}{c = cpi} and \texorpdfstring{$\bm{c_{\pi}}$}{cpi}-monotonicity}\label{sec:spread}
When the radii are spread apart, a reasonable approximate cost
appears to be
\[
c_\pi(r_1, r_2, r_3) =
\frac{1}{r_1+r_2} + \frac{1}{r_2+r_3} + \frac{1}{r_3-r_1},
\]
which arises from collocating the charges at angles $\theta_2=\pi$ and
$\theta_3=0$ (see \autoref{fig:three-charges-spread-apart}).
In the first part of this section we want to study under which condition on the radii $r_1$, $r_2$ and
$r_3$ we have
\[
c(r_1, r_2, r_3) = c_\pi(r_1, r_2, r_3).
\]
We start with a heuristic argument involving a necessary condition.
Up to permutations, we may assume $r_1 \leq r_2 \leq r_3$. It is simple to check that
\[
C_\theta(r_1, r_2, r_3, \pi, 0) = 0,
\]
where $C$ has been defined in \eqref{eq:coulomb}, either by direct computation or by a symmetry argument.\footnote{In fact, the four
configurations with $\theta_2,\theta_3\in\{0,\pi\}$ are always stationary.}
If $(\theta_2, \theta_3) = (\pi, 0)$ must be a minimum, then a
necessary condition is
\[
C_{\theta\theta}(r_1, r_2, r_3, \pi, 0) \geq 0,
\]
in the sense that the Hessian matrix is positive-definite.
We have
\[
C_{\theta\theta}(r_1, r_2, r_3, \pi, 0) =
\begin{pmatrix}
r_2 \left( \frac{r_1}{(r_1+r_2)^3} + \frac{r_3}{(r_2+r_3)^3} \right) &
-\frac{r_2 r_3}{(r_2+r_3)^3} \\
-\frac{r_2 r_3}{(r_2+r_3)^3} &
r_3 \left( \frac{r_2}{(r_2+r_3)^3} - \frac{r_1}{(r_3-r_1)^3} \right);
\end{pmatrix}
\]
since the first entry is positive, this $2\times2$ matrix is positive-definite if and only if the determinant is positive too, namely
\[
\det C_{\theta\theta}(r_1, r_2, r_3, \pi, 0)  = -\frac{r_1 r_2 r_3 [r_2r_3(r_2-r_3) + r_1(r_2^2 + 5r_2r_3 + r_3^2) + r_1^3]}
{(r_1+r_2)^3 (r_2+r_3)^2 (r_3-r_1)^3} \geq 0,
\]
or equivalently\[
r_1(r_2^2 + 5r_2r_3 + r_3^2) + r_1^3 < r_2r_3(r_3-r_2).
\]
\autoref{fig:spread} depicts the region where the Hessian is positive.

\begin{figure}
\centering
\includegraphics[scale=1]{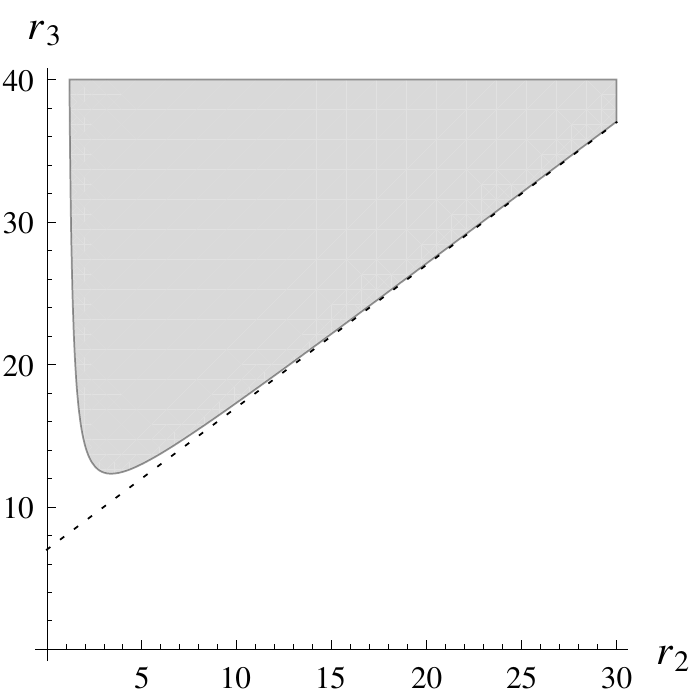}
\caption{The region in the $(r_2,r_3)$ plane where
$C_{\theta\theta}(r_1, r_2, r_3, \pi, 0) \geq 0$, with $r_1=1$.
The dotted line is $r_3=r_2+7$.}
\label{fig:spread}
\end{figure}
We partially justify the previous argument in the following lemma
which, despite not being quantitative, will suffice for our purposes.

\begin{lemma}\label{lemma:c-pi}
If $0<r_1^-\le r_1^+<r_2^-\le r_2^+$, then there exists $r_3^-(r_1^-,r_1^+,r_2^-,r_2^+)$ such that
for every $r_1\in[r_1^-,r_1^+]$, $r_2\in[r_2^-,r_2^+]$ and $r_3\geq r_3^-$ we have
\[
c(r_1, r_2, r_3) = c_\pi(r_1, r_2, r_3).
\]
\end{lemma}
\begin{proof}
We denote by $\setT^2$ the $2$-dimensional torus $\setR^2 / (2\pi \setZ)^2$. The idea of the proof is the following: we claim that for sufficiently large $r_3$ there are exactly four stationary points $(\theta_2,\theta_3)\in\setT^2$ for $C(r_1,r_2,r_3,\theta_2,\theta_3)$, corresponding to $\theta_2,\theta_3\in\{0,\pi\}$. Therefore $c(r_1,r_2,r_3)$ must coincide with the value achieved at one of them and by comparing the four values we arrive at the desired conclusion.

First of all, we compute the gradient
\[
C_\theta(r_1,r_2,r_3,\theta_2,\theta_3) = \begin{pmatrix}
-\frac{r_1r_2\sin(\theta_2)}{\left(r_1^2+r_2^2-2r_1r_2\cos(\theta_2)\right)^{3/2}}
-\frac{r_2r_3\sin(\theta_2-\theta_3)}
	{\left(r_2^2+r_3^2-2r_2r_3\cos(\theta_2-\theta_3)\right)^{3/2}} \\
-\frac{r_1r_3\sin(\theta_3)}{\left(r_1^2+r_3^2-2r_1r_3\cos(\theta_3)\right)^{3/2}}
+\frac{r_2r_3\sin(\theta_2-\theta_3)}
	{\left(r_2^2+r_3^2-2r_2r_3\cos(\theta_2-\theta_3)\right)^{3/2}}
\end{pmatrix}.
\]
The gradient vanishes if and only if the following equations are simultaneously satisfied:
\begin{subequations}
\begin{gather}
\label{eq:vertical-curve}
\frac{r_1r_2\sin(\theta_2)}{\left(r_1^2+r_2^2-2r_1r_2\cos(\theta_2)\right)^{3/2}} +
\frac{r_1r_3\sin(\theta_3)}{\left(r_1^2+r_3^2-2r_1r_3\cos(\theta_3)\right)^{3/2}} = 0, \\
\label{eq:diagonal-curve}
-\frac{r_1r_3\sin(\theta_3)}{\left(r_1^2+r_3^2-2r_1r_3\cos(\theta_3)\right)^{3/2}}
+\frac{r_2r_3\sin(\theta_2-\theta_3)}
	{\left(r_2^2+r_3^2-2r_2r_3\cos(\theta_2-\theta_3)\right)^{3/2}} = 0.
\end{gather}
\end{subequations}
To show that there are exactly four stationary points,
the idea is that, for $r_3$ sufficiently large, equations \eqref{eq:vertical-curve}
and \eqref{eq:diagonal-curve} define two pairs of closed curves on $\setT^2$, of type
$(0,1)$ and $(1,1)$ respectively, with the property that every curve from the first
family intersects each curve of the second family in a single point.
The situation is represented in \autoref{fig:curves}.

\begin{figure}
\centering
\includegraphics[scale=1]{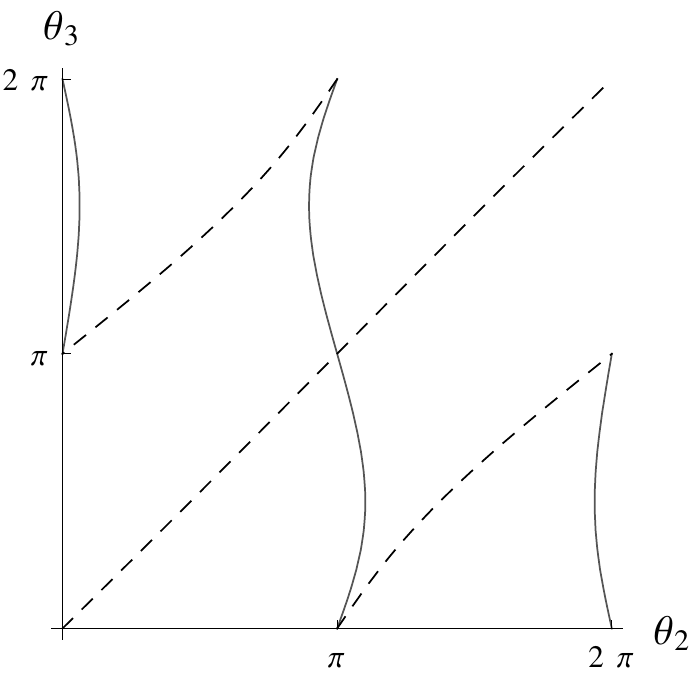}
\caption{The curves in $\setT^2$ whose four intersections correspond to
stationary points of $C(r_1,r_2,r_3,\theta_2,\theta_3)$.
The two solid curves are defined by \eqref{eq:vertical-curve}.
The dashed curves are defined by \eqref{eq:diagonal-curve}.}
\label{fig:curves}
\end{figure}

\step{1}
Given $r_1$, $r_2$ and a sufficiently large $r_3$,
we claim that for every $\theta_3\in S^1$ there are exactly two values
$\tilde\theta_2^0(\theta_3), \tilde\theta_2^\pi(\theta_3) \in S^1$ which satisfy
\eqref{eq:vertical-curve}; moreover $\tilde\theta_2^0(\theta_3)$ and
$\tilde\theta_2^\pi(\theta_3)$ are close to $0$ and $\pi$ respectively by less than
$O\bigl(r_3^{-2}\bigr)$, uniformly in $\theta_3$, and their derivatives go to
to zero uniformly in $\theta_3$ for $r_3\to\infty$.\footnote{More precisely, they
are close to zero by less than $O\bigl(r_3^{-2}\bigr)$, uniformly in $\theta_3$.}
These functions correspond to the solid, almost vertical, lines in
\autoref{fig:curves}.

We begin by finding a useful bound on $\abs{\sin(\theta_2)}$.
The two terms of \eqref{eq:vertical-curve} can be estimated by
\begin{align*}
\abs{\frac{r_1r_2\sin(\theta_2)}
	{\left(r_1^2+r_2^2-2r_1r_2\cos(\theta_2)\right)^{3/2}}} &\geq
	\frac{r_1^-r_2^-\abs{\sin(\theta_2)}}{(r_1^++r_2^+)^3}, \\
\abs{\frac{r_1r_3\sin(\theta_3)}
	{\left(r_1^2+r_3^2-2r_1r_3\cos(\theta_3)\right)^{3/2}}} &\leq
	\frac{r_1^+r_3}{(r_3-r_1^+)^3},
\end{align*}
therefore, in order to have equality \eqref{eq:vertical-curve}, it must be that
\[
\frac{r_1^-r_2^-\abs{\sin(\theta_2)}}{(r_1^++r_2^+)^3} \leq
\frac{r_1^+r_3}{(r_3-r_1^+)^3},
\]
that is
\begin{equation}\label{eq:sin-theta2}
\abs{\sin(\theta_2)} \leq
\frac{r_1^+(r_1^++r_2^+)^3}{r_1^-r_2^-} \cdot \frac{r_3}{(r_3-r_1^+)^3} =
O\bigl(r_3^{-2}\bigr)
\end{equation}
as $r_3\to\infty$, where the implied constant depends only on $r_1^\pm$ and
$r_2^\pm$.

We have already discussed that, for every $\theta_3\in S^1$, the second term in
\eqref{eq:vertical-curve} is smaller than $r_3 (r_3-r_1^+)^{-3}$ in magnitude.
On the other hand, the first term vanishes for $\theta_2=0,\pi$ and is equal to
$\pm r_1r_2(r_1^2+r_2^2)^{3/2}$ for $\theta_2=\pm\pi/2$. Therefore, by continuity,
for $r_3$ large we have at least two solutions to \eqref{eq:vertical-curve}.

The estimate on $\abs{\sin(\theta_2)}$ proves that the solutions must be located
near $0$ and $\pi$. Now we want to prove that there are exactly two of them.
To do so, we verify that the partial derivative with respect to $\theta_2$ of
the first term in \eqref{eq:vertical-curve} is different from zero for $\theta_2$ in
the prescribed intervals around $0$ and $\pi$. Indeed, the derivative is
\begin{align*}
\frac{\partial}{\partial\theta_2} \biggr\rvert_{\theta_2=0} \left(
\frac{r_1r_2\sin(\theta_2)}
	{\left(r_1^2+r_2^2-2r_1r_2\cos(\theta_2)\right)^{3/2}} \right) &=
\frac{r_1r_2}{(r_2-r_1)^3}, \\
\frac{\partial}{\partial\theta_2} \biggr\rvert_{\theta_2=\pi} \left(
\frac{r_1r_2\sin(\theta_2)}
	{\left(r_1^2+r_2^2-2r_1r_2\cos(\theta_2)\right)^{3/2}} \right) &=
-\frac{r_2}{(r_1+r_2)^3},
\end{align*}
therefore it is different from zero around the two points and the two solutions are
simple.

The claim is almost entirely proved. We now have the two functions
$\tilde\theta_2^0(\plchldr), \tilde\theta_2^\pi(\plchldr)$ and
the last thing that we want to derive is the estimate of their first derivatives.
Let $\theta_2(\plchldr)$ be one of the two functions. Thanks to the implicit function
theorem, we know that $\theta_2(\plchldr)$ is at least $C^1$ and we can compute
\begin{multline*}
\theta_2'(\theta_3) =
-\frac{r_3}{r_2} \cdot
\frac{2(r_1^2+r_3^2)\cos(\theta_3)+r_1r_3[-5+\cos(2\theta_3)]}
	{2(r_1^2+r_2^2)\cos\bigl(\theta_2(\theta_3)\bigr)+
	r_1r_2\bigl[-5+\cos\bigl(2\theta_2(\theta_3)\bigr)\bigr]} \\
\cdot \left( \frac{r_1^2+r_2^2-2r_1r_2\cos\bigl(\theta_2(\theta_3)\bigr)}
	{r_1^2+r_3^2-2r_1r_3\cos(\theta_3)} \right)^{5/2}.
\end{multline*}
All the terms are fairly easy to deal with, apart from the denominator of the second
fraction. However, we have that
\begin{align*}
2(r_1^2+r_2^2)\cos(\theta_2) + r_1r_2[-5+\cos(2\theta_2)] \bigr\rvert_{\theta_2=0}
&= 2(r_1^2-2r_1r_2+r_2^2) \geq 2(r_2^--r_1^+)^2, \\
-2(r_1^2+r_2^2)\cos(\theta_2) - r_1r_2[-5+\cos(2\theta_2)] \bigr\rvert_{\theta_2=\pi}
&= 2(r_1^2+2r_1r_2+r_2^2) \geq 2(r_2^-+r_1^-)^2,
\end{align*}
therefore, by the continuity of the functions involved and by compactness,
there exists a neighbourhood $U$ of $\{0,\pi\}$ such that if $r_1\in[r_1^-,r_1^+]$,
$r_2\in[r_2^-,r_2^+]$ and $\theta_2\in U$ then
\[
\abs{2(r_1^2+r_2^2)\cos(\theta_2) + r_1r_2[-5+\cos(2\theta_2)]} >
(r_2^--r_1^+)^2.
\]
From this and \eqref{eq:sin-theta2}, which ensures that $\theta_2(\theta_3)\in U$,
we deduce that for $r_3$ large
\[
\abs{\theta_2'(\theta_3)} \leq
\frac{r_3}{r_2^-} \cdot
\frac{2(r_1^+)^2+2r_3^2}{(r_2^--r_1^+)^2} \cdot
\frac{(r_1^++r_2^+)^5}{(r_3-r_1^+)^5} = O\left(r_3^{-2}\right).
\]


\step{2}
Next we perform the same analysis for \eqref{eq:diagonal-curve}.
We prove that there exist two $C^1$ functions $\hat\theta_2^0(\theta_3)$ and
$\hat\theta_2^\pi(\theta_3)$ which are the only solutions of \eqref{eq:diagonal-curve}
when $\theta_3$ is prescribed and that their derivatives are strictly positive.
First of all, we introduce the new variable $\psi = \theta_2-\theta_3$.
Equation \eqref{eq:diagonal-curve} reads as
\begin{equation}\label{eq:diagonal-curve-psi}
-\frac{r_1\sin(\theta_3)}{\left(r_1^2+r_3^2-2r_1r_3\cos(\theta_3)\right)^{3/2}}
+\frac{r_2\sin(\psi)}{\left(r_2^2+r_3^2-2r_2r_3\cos(\psi)\right)^{3/2}} = 0.
\end{equation}

\begin{itemize}
\item \textbf{The solutions lie in two strips.}
From equation \eqref{eq:diagonal-curve-psi} we get
\[
\begin{split}
\frac{r_1^+}{(r_3-r_1^+)^3} &\geq
\abs{\frac{r_1\sin(\theta_3)}{\left(r_1^2+r_3^2-2r_1r_3\cos(\theta_3)\right)^{3/2}}}\\
&= \abs{\frac{r_2\sin(\psi)}{\left(r_2^2+r_3^2-2r_2r_3\cos(\psi)\right)^{3/2}}} \geq
\frac{r_2^-\abs{\sin(\psi)}}{(r_2^++r_3)^3}.
\end{split}
\]
Therefore we have
\[
\abs{\sin(\psi)} \leq \left(\frac{r_3+r_2^+}{r_3-r_1^+}\right)^3 \frac{r_1^+}{r_2^-},
\]
which, for $r_3$ sufficiently large, implies $\abs{\sin(\psi)} < \eta$
for a fixed $\eta\in(r_1^+/r_2^-,1)$.
\item \textbf{There are at least two solutions.}
The first term of \eqref{eq:diagonal-curve-psi} is bounded by
\[
\abs{\frac{r_1\sin(\theta_3)}{\left(1+r_3^2-2r_3\cos(\theta_3)\right)^{3/2}}} \leq
\frac{r_1}{(r_3-1)^3}.
\]
On the other hand, when $\psi=\pm\pi/2$ the second term equals
\[
\pm \frac{r_2}{(r_2^2+r_3^2)^{3/2}},
\]
which is bigger for $r_3$ large enough.
This tells us that for every $\theta_3$ there are at least two distinct values of
$\psi$ which solve \eqref{eq:diagonal-curve-psi}, because the second term is a
continuous periodic function of $\psi$.
\item \textbf{There are exactly two solutions.} The derivative of the second term is
\[
\frac{\partial}{\partial\psi}\left(
	\frac{r_2\sin(\psi)}{(r_2^2+r_3^2-2r_2r_3\cos(\psi))^{3/2}}
	\right) =
\frac{-3r_2^2r_3+(r_2^3+r_2r_3^2)\cos(\psi)+r_2^2r_3\cos(\psi)^2}
	{(r_2^2+r_3^2-2r_2r_3\cos(\psi))^{5/2}}.
\]
We observe that the denominator is always positive. We study the sign of the
numerator. The equation
\[
-3r_2^2r_3+(r_2^3+r_2r_3^2)t+r_2^2r_3t^2 = 0
\]
for the unknown $t$ has the two solutions
\[
\frac{-r_2^2-r_3^2+\sqrt{r_2^4+14r_2^2r_3^2+r_3^4}}{2r_2r_3}, \qquad
\frac{-r_2^2-r_3^2-\sqrt{r_2^4+14r_2^2r_3^2+r_3^4}}{2r_2r_3}.
\]
However, only the first one lies in the range $[-1,1]$, whereas the second is
less than $-2$. In fact,
\[
r_2^2+r_3^2+\sqrt{r_2^4+14r_2^2r_3^2+r_3^4} \geq
r_2^2+r_3^2+\sqrt{r_2^4+2r_2^2r_3^2+r_3^4} =
2(r_2^2+r_3^2) \geq 4r_2r_3.
\]
Therefore the function has exactly two stationary points and is monotone between them.
\item \textbf{Derivative of the solutions.}
At this point we know that there exist two functions $\psi_0(\theta_3)$ and
$\psi_\pi(\theta_3)$ such that the corresponding
$\hat\theta_2^0(\theta_3)=\psi_0(\theta_3)+\theta_3$ and
$\hat\theta_2^\pi(\theta_3)=\psi_\pi(\theta_3)+\theta_3$
parametrize the solutions of \eqref{eq:diagonal-curve}.

The goal is to show that for $r_3$ sufficiently large we have
$\theta_2'(\theta_3) \geq C > 0$ for some constant $C$ independent of $r_3$,
where $\theta_2(\plchldr)$ is either $\hat\theta_2^0(\plchldr)$ or
$\hat\theta_2^\pi(\plchldr)$.
Thanks to the implicit function theorem we can compute the derivative
\begin{multline*}
\theta_2'(\theta_3) =
\frac{(r_2^2+r_3^2-2r_2r_3\cos(\psi))^{5/2}}
	{-3r_2^2r_3+(r_2^3+r_2r_3^2)\cos(\psi)+r_2^2r_3\cos(\psi)^2} \\
\cdot \Biggl(
\frac{r_1\cos(\theta_3)}{(r_1^2+r_3^2-2r_1r_3\cos(\theta_3))^{3/2}} +
\frac{r_2\cos(\psi)}{(r_2^2+r_3^2-2r_2r_3\cos(\psi))^{3/2}} \\
-\frac{3r_1^2r_3\sin(\theta_3)^2}{(r_1^2+r_3^2-2r_1r_3\cos(\theta_3))^{5/2}} -
\frac{3r_2^2r_3\sin(\psi)^2}{(r_2^2+r_3^2-2r_2r_3\cos(\psi))^{5/2}}
\Biggr),
\end{multline*}
where $\psi=\theta_2-\theta_3$ as before.
We introduce the parameter $\kappa=1/r_3$ and write the derivative in terms of it.
We have that
\[
\theta_2'(\theta_3) = f(r_1,r_2,1/r_3,\theta_2-\theta_3,\theta_3)
\]
where
\begin{multline}\label{eq:diagonal-curve-derivative}
f(r_1,r_2,\kappa,\psi,\theta_3) =
\frac{(1-2r_2\kappa\cos(\psi)+r_2\kappa^2)^{5/2}}
	{-3r_2\kappa+(r_2^3\kappa^2+r_2)\cos(\psi)+r_2^2\kappa\cos(\psi)^2} \\
\cdot \Biggl(
\frac{r_1\cos(\theta_3)}{(1+r_1^2\kappa^2-2r_1\kappa\cos(\theta_3))^{3/2}} +
\frac{r_2\cos(\psi)}{(1+r_2^2\kappa^2-2r_2\kappa\cos(\psi))^{3/2}} \\
-\frac{3r_1^2\kappa\sin(\theta_3)^2}{(1+r_1^2\kappa^2-2r_1\kappa\cos(\theta_3))^{5/2}}
-\frac{3r_2^2\kappa\sin(\psi)^2}{(1+r_2^2\kappa^2-2r_2\kappa\cos(\psi))^{5/2}}
\Biggr).
\end{multline}
Observe that the only singularities are due to the denominator of the first fraction.
However, the singular values of $\psi$ lie outside the two intervals
\[
S = [-\arcsin(\eta),\arcsin(\eta)]\cup[\pi-\arcsin(\eta),\pi+\arcsin(\eta)]
\]
for $\kappa$ sufficiently small ($r_3$ large enough), because they converge to
$\pm\pi/2$. Therefore there exists $\kappa^+>0$ such that the function $f$ is
continuous in the domain
\[
D = [r_1^-,r_1^+]_{r_1} \times [r_2^-,r_2^+]_{r_2} \times [0,\kappa^+]_\kappa
\times S_\psi \times [0,2\pi]_{\theta_3}.
\]
\item \textbf{Limit case.}
We rewrite equation \eqref{eq:diagonal-curve-psi} in terms of $\kappa$ as
\begin{equation}\label{eq:diagonal-curve-kappa}
-\frac{r_1\sin(\theta_3)}{(1+r_1^2\kappa^2-2r_1\kappa\cos(\theta_3))^{3/2}} +
\frac{r_2\sin(\psi)}{(1+r_2^2\kappa^2-2r_2\kappa\cos(\psi))^{3/2}} = 0.
\end{equation}
Let $\Gamma_{r_1,r_2,\kappa}$ denote the set of solutions
$(\psi,\theta_3)\in S_\psi\times[0,2\pi]_{\theta_3}$ to \eqref{eq:diagonal-curve-kappa}.
By the continuity of \eqref{eq:diagonal-curve-kappa} we know that
\[
\Gamma = \bigcup_{r_1\in[r_1^-,r_1^+]} \bigcup_{r_2\in[r_2^-,r_2^+]}
\bigcup_{\kappa\in[0,\kappa^+]\vphantom{r_1^+}} \Gamma_{r_1,r_2,\kappa} \subset D
\]
is a closed set. Our ultimate goal is to show that $f$ is positive on
$\Gamma_{r_1,r_2,\kappa}$ when $\kappa$ is small enough.

We start by studying the limit case $\kappa=0$.
The limit curve $\Gamma_{r_1,r_2,0}$ is given by the equation
\begin{equation}\label{eq:limit-curve}
r_1 \sin(\theta_3) = r_2 \sin(\psi).
\end{equation}
For $\kappa=0$, the function $f$ equals
\[
f(r_1,r_2,0,\psi,\theta_3) =
\frac{1}{r_2\cos(\psi)}\bigl(r_1\cos(\theta_3)+r_2\cos(\psi)\bigr) =
1 + \frac{r_1\cos(\theta_3)}{r_2\cos(\psi)}.
\]
We claim that this function is positive on the curve defined by
\eqref{eq:limit-curve}.
Indeed, positivity is guaranteed if we are able to prove that
\[
\abs{\frac{r_1\cos(\theta_3)}{r_2\cos(\psi)}} < 1.
\]
But, by squaring, this is equivalent to
\[
r_1^2 \cos(\theta_3)^2 < r_2 \cos(\psi)^2,
\]
which, thanks to \eqref{eq:limit-curve}, reduces to the true inequality $r_1^2<r_2^2$.
\item \textbf{Conclusion.} Finally, we prove that $f\geq C>0$ on $\Gamma_{r_1,r_2,\kappa}$
for $\kappa$ close to zero, where $C$ is a constant depending only on $r_1^\pm$ and
$r_2^\pm$.

We know that $f$ is positive on the compact set
\[
K = \bigcup_{r_1\in[r_1^-,r_1^+]} \bigcup_{r_2\in[r_2^-,r_2^+]} \Gamma_{r_1,r_2,0}.
\]
Therefore there exists a positive constant $C$ and an open neighbourhood $U$ of $K$ in $D$
such that $f>C$ on $U$. Since $\Gamma$ is closed, a compactness argument shows that
$\Gamma_{r_1,r_2,\kappa}\subset U$ for $\kappa$ close to zero and this concludes the
proof.
\end{itemize}

\step{3}
The previous steps tell us that \eqref{eq:vertical-curve} defines two vertical curves and
\eqref{eq:diagonal-curve} two diagonal curves.
The estimates on the derivatives of such curves prove that the intersections are
simple, therefore there are exactly four stationary points. But we already know four
stationary points, namely
\[
(\theta_2,\theta_3) = (0,0),\ (0,\pi),\ (\pi,0),\ (\pi,\pi).
\]
To conclude, we can just compare the costs associated to each of them and pick the
smallest one. It is easy to see that $(\theta_2,\theta_3)=(\pi,0)$ is the optimal
choice. In fact, $(0,0)$ is clearly the worst. Among the three cases left,
we can say that $(\pi,0)$ always beats $(\pi,\pi)$, that is
\begin{multline*}
C(r_1,r_2,r_3,\pi,\pi) - C(r_1,r_2,r_3,\pi,0) \\
= \left( \frac1{r_3-r_2}-\frac1{r_3-r_1} \right) +
\left( \frac1{r_2+r_1}-\frac1{r_3+r_2} \right) > 0,
\end{multline*}
as both the differences in parenthesis are positive.
Finally, $(\pi,0)$ beats $(0,\pi)$ too because
\[
C(r_1,r_2,r_3,0,\pi) - C(r_1,r_2,r_3,\pi,0) =
\frac{2r_1(r_3^2-r_2^2)}{(r_2^2-r_1^2)(r_3^2-r_1^2)} > 0. \qedhere
\]
\end{proof}

In the following lemma, we prove that, with the frozen cost $c_\pi$, given six increasing radii numbered $1,\dotsc,6$ the choice of two disjoint subsets of three elements which minimizes the cost is always given by $145$ and $236$. Actually, we prove only some comparisons that are enough for our examples, but one could show in general that
\begin{multline*}
c_\pi(r_1, r_4, r_5) + c_\pi(r_2, r_3, r_6) = \\
= \min \Set{ c_\pi(p_1,p_2,p_3) + c_\pi(p_4,p_5,p_6)
	: \{p_1,\dotsc,p_6\} = \{r_1,\dotsc,r_6\} }.
\end{multline*}
The proof of this fact reduces to the characterization of $c$-monotonicity with Coulomb cost performed in \cite[Proposition 2.4]{CDD}.
\begin{lemma}\label{lemma:c-pi145}
Let $0 < r_1 < \dotsb < r_6$. Then we have that
\begin{equation}\label{eqn:c-pi-monot}
\begin{split}
 c_\pi(r_1, r_4, r_5) +c&_\pi(r_2, r_3, r_6)
  \leq \min\big\{ c_\pi(r_1, r_4, r_6) + c_\pi(r_2, r_3, r_5) ,\\
 &  c_\pi(r_1, r_3, r_6) + c_\pi(r_2, r_4, r_5),\;c_\pi(r_1, r_3, r_5) + c_\pi(r_2, r_4, r_6) \big\}.
 \end{split}
\end{equation}
\end{lemma}
\begin{proof}
Let us consider the one dimensional Coulomb cost defined in $\setR$
$$\bar c(v_1,v_2,v_3) = \frac{1}{\abs{v_2-v_1}} + \frac{1}{\abs{v_3-v_2}} + \frac{1}{\abs{v_1-v_3}} \qquad \forall v_1,v_2,v_3\in \setR.$$
We notice that $ c_\pi(r_1, r_4, r_5) = \bar c(r_1, -r_4, r_5) $ and, more in general, for all the $3$-uples appearing in \eqref{eqn:c-pi-monot} the $c_\pi$-cost and the $\bar c$-cost satisfy the same relation.
In \cite[Proposition 2.4]{CDD} it is proved that, given the six points $-r_4, -r_3, r_1,r_2,r_5,r_6$ the best way to choose two $3$-uples to minimize the one dimensional Coulomb cost is to take the points in odd position and the points in even position; in particular, we have
\begin{equation*}
\begin{split}
 \bar c( -r_4, r_1, r_5) + \bar c(-&r_3, r_2,  r_6)
  \leq \min\big\{ \bar c(-r_4, r_1, r_6) + \bar c(-r_3, r_2, r_5) ,\\
 &  \bar c(-r_3, r_1, r_6) + \bar c(-r_4, r_2, r_5), \; \bar c(-r_3, r_1, r_5) + \bar c(-r_4, r_2, r_6) \big\},
  \end{split}
\end{equation*}
which proves \eqref{eqn:c-pi-monot}.
\end{proof}
\begin{remark}
The previous lemma allows to prove that, for the cost $c_\pi$, the symmetrized optimal plan for the problem~\eqref{kantorovich} is unique and coincides with the symmetrization of the $DDI$-map.
\end{remark}

\section{Proofs of examples and counterexamples}\label{sec:proofs}
\begin{proof}[Proof of \autoref{ce:145}]
Let $t_0$ be given by \autoref{lemma:120asint} and let us choose $\eps \leq t_0/2$. If, by contradiction, the $DDI$-map $T$ associated to $\rho_\eps$ is optimal, by \autoref{monotonia} its support is $c$-monotone. Let us consider $1+ \eps$, $1+3\eps$ and the images of these points through $T$ and $ T\circ T$:
$$T(1+ \eps) = 1+ 7\eps, \quad T\circ T(1+ \eps) = 1+ 9\eps,$$
$$T(1+ 3\eps) = 1+ 5\eps, \quad T\circ T(1+ 3\eps) = 1+ 11\eps,$$
 We notice that these points 
 $$(r_1,...,r_6) = (1+\eps, 1+3\eps, 1+5\eps, 1+7\eps, 1+9\eps, 1+11\eps),$$
 are equally spaced; hence, we can apply the scaling properties of the cost function and \autoref{lemma:120asint} with $t= 2\eps/(1+\eps) \leq t_0$ to deduce that,
\begin{equation*}
\begin{split}
c(r_1, r_4, r_6) + c(r_2, r_3, r_5) &= \frac{1}{1+\eps}\Big[c\Big(\frac{r_1}{1+\eps}, \frac{r_4}{1+\eps}, \frac{r_6}{1+\eps}\Big) + c\Big(\frac{r_2}{1+\eps}, \frac{r_3}{1+\eps}, \frac{r_5}{1+\eps}\Big) \Big]
\\&<\frac{1}{1+\eps}\Big[c\Big(\frac{r_1}{1+\eps}, \frac{r_4}{1+\eps}, \frac{r_5}{1+\eps}\Big) + c\Big(\frac{r_2}{1+\eps}, \frac{r_3}{1+\eps}, \frac{r_6}{1+\eps}\Big) \Big]
\\&=  c(r_1, r_4, r_5) + c(r_2, r_3, r_6).
\end{split}
\end{equation*}
This contradicts the $c$-monotonicity of the support by taking $p=\{ 3\}$.
\end{proof}

\begin{proof}[Proof of \autoref{ce:class}]
\step{1} By choosing $\eps$ sufficiently small (independently on $M$),
we exclude that the $DDI$-map is optimal in problem~\eqref{monge} for every $M>2$.

Let $T$ be the piecewise continuous $DDI$-map.
Consider the following two points in the support of the plan associated to $T$
(recall that the support is a closed set):
\begin{align*}
\left(1+\frac\eps2,T\left(1+\frac\eps2\right),T^{(2)}\left(1+\frac\eps2\right)\right) &=
	\left(1+\frac\eps2,1+\frac{7\eps}2,1+\frac{9\eps}2\right), \\
\lim_{r\to1+\eps^-}\bigl(r,T(r),T^{(2)}(r)\bigr) &= (1+\eps,1+3\eps,1+5\eps).
\end{align*}
We claim that they violate the $c$-monotonicity property (\autoref{monotonia})
with $p=\{3\}$, namely
\[
\begin{split}
f(\eps) =
{}& c\left(1+\frac\eps2,1+\frac{7\eps}2,1+\frac{9\eps}2\right)
	+ c(1+\eps,1+3\eps,1+5\eps) \\
&-\left[c\left(1+\frac\eps2,1+\frac{7\eps}2,1+5\eps\right)
	+c\left(1+\eps,1+3\eps,1+\frac{9\eps}2\right)\right] > 0
\end{split}
\]
for $\eps$ sufficiently small. The proof is similar to that of \autoref{lemma:120asint}. Using the formulas obtained in \autoref{lemma:derivatives}
we just compute the derivatives
\begin{gather*}
f(0) = f'(0) = f''(0) = 0, \\
f'''(0) = \frac{71\sqrt3}{100} > 0.
\end{gather*}

\step{2} We exclude that the maps $DID$, $IDD$, $III$ in the $\{D,I\}^3$-class are optimal in problem~\eqref{monge} for $M$ large enough.

We present the argument to exclude the $DID$-map, the others being similar. Let us fix $x,y \in (M+1/4, M+3/4)$, $x<y$, and let us consider their orbits through $T$, that is $T(x),T(y) \in (1, 1+\eps_0)$ and $T^{(2)}(x),T^{(2)}(y) \in (1+ 3\eps_0, 1+4\eps_0)$.
Let us consider the increasingly ordered points 
$$(r_1,...,r_6) = \Big(T(y), T(x),T^{(2)}(x),T^{(2)}(y), x,y \Big);$$
the couples of points $(r_1,r_4,r_6)$ and $(r_2,r_3,r_5)$ belong to the support of the plan associated to the $DID$-map. 
By \autoref{lemma:c-pi}, we can choose $M$ sufficiently large so that the previous points, as well as the points $(r_1, r_4, r_5)$ and $(r_2, r_3, r_6)$, have the same $c$ and $c_\pi$ cost. By \autoref{lemma:c-pi145}, which describes the $c_{\pi}$ monotonicity, we have 
\begin{equation*}
\begin{split}
c(r_1, r_4, r_5) + c(r_2, r_3, r_6)&= c_\pi(r_1, r_4, r_5) + c_\pi(r_2, r_3, r_6) \\
& \leq c_\pi(r_1, r_4, r_6) + c_\pi(r_2, r_3, r_5) \\
& = c(r_1, r_4, r_6) + c(r_2, r_3, r_5).
\end{split}
\end{equation*}
This shows, by \autoref{monotonia}, that the $DID$-map cannot be optimal.\end{proof}

\begin{remark}\label{ce:4marg}
Our method can be applied to the $4$-marginal problem to show that
there exists $\eps>0$ such that, setting 
\[
\rho_\eps = \frac1{16\eps} 1_{[1,1+16\eps]} \,dr \in \Meas(\setR_+),
\]
any map in the $\{D,I\}^4$-class associated to $\rho_\eps$ is not optimal
in problem~\eqref{monge}. Indeed, let $T$ be any such map.
Pick two points in $[1,1+16\eps]$ such that the union of their two orbits is
\[
\{r_1,\dotsc,r_8\} =
\{1+\eps,1+3\eps,1+5\eps,1+7\eps,1+9\eps,1+11\eps,1+13\eps,1+15\eps\}.
\]
We claim that $T$ is not $c$-monotone because the partitioning of
$\{r_1,\dotsc,r_8\}$ into two quartets that minimizes
\[
c(r_{i_1},r_{i_2},r_{i_3},r_{i_4}) + c(r_{i_5},r_{i_6},r_{i_7},r_{i_8})
\]
is $\{(r_1,r_5,r_6,r_7),(r_2,r_3,r_4,r_8)\}$ and
such partition doesn't correspond to any of the maps in the $\{D,I\}^4$-class.

The way to see this is to extend the results of Section~\ref{sec:taylor} to the
$4$-marginal case. Consider four radii
\[
(r_1,r_2,r_3,r_4) = (1+a_1t,1+a_2t,1+a_3t,1+a_4t).
\]
Following the same derivation, we find that the angles that give the cost $c$ are
\[
\begin{pmatrix}
\theta_2(t) \\ \theta_3(t) \\ \theta_4(t)
\end{pmatrix} =
\begin{pmatrix}
\pi/2 \\ \pi \\ 3/3\pi
\end{pmatrix}
+ \frac{6-\sqrt2}{34}
\begin{pmatrix}
-a_1-a_2+a_3+a_4 \\ 2a_4-2a_2 \\ a_1-a_2-a_3+a_4
\end{pmatrix} t + o(t).
\]
In turn, this provides the expansion of the cost up to the third order and this
information can be used to verify the asymptotic optimality of any given partition.
We omit the formulas, since this computations are better performed with the aid
of a computer algebra system.
\end{remark}

\begin{proof}[Proof of \autoref{prop:example}]
Let $M$ be chosen, thanks to \autoref{lemma:c-pi}, so that
\begin{equation}
\label{eqn:ex-M-choice}
c(r_1, r_2, r_3) = c_\pi(r_1, r_2, r_3) \qquad \mbox{for every } r_1\in [1,2], \; r_2 \in [3,4], \; r_3 \in [M,\infty).
\end{equation}
In order to prove the $c$-monotonicity property, since the map $T$ is cyclical and since its orbits take exactly one point in each interval $[1,2]$, $[3,4]$, and $[M,\infty)$, it is enough to show that, given $x, y\in [1,2]$, $x<y$, we have
\begin{equation}
\label{eqn:c-monot-example}
c\bigl(x,T(x), T^{(2)}(x)\bigr)+ c\bigl(y,T(y), T^{(2)}(y)\bigr) \leq c(x,A,B)+ c(y,C,D)
\end{equation}
for every possible choice of $A,B,C,D$ such that $\{A , C\} = \{ T(x), T(y)\}$ and $\{B , D\} = \{ T^{(2)}(x), T^{(2)}(y)\}$. 
By definition, we have that 
$$1\leq x<y\leq 2 \leq 3 \leq T(y) <T(x) \leq 4 \leq M \leq T^{(2)}(x) <T^{(2)}(y);$$
hence by \eqref{eqn:ex-M-choice} we have that $c\bigl(x,T(x), T^{(2)}(x)\bigr) = c_\pi\bigl(x,T(x), T^{(2)}(x)\bigr)$ (and similarly for $y$ and for the other $3$-uples) and by \autoref{lemma:c-pi145} we have that
\begin{equation*}
\begin{split}
c\bigl(x,T(x), T^{(2)}(x)\bigr)&+c\bigl(y,T(y), T^{(2)}(y)\bigr)
= c_\pi\bigl(x,T(x), T^{(2)}(x)\bigr)+ c_\pi\bigl(y,T(y), T^{(2)}(y)\bigr) \\
&\leq c_\pi(x,A,B)+ c_\pi(y,C,D) = c(x,A,B)+ c(y,C,D),
\end{split}
\end{equation*}
for every possible choice of $A,B,C,D$ such that $\{A , C\} = \{ T(x), T(y)\}$ and $\{B , D\} = \{ T^{(2)}(x), T^{(2)}(y)\}$; this proves \eqref{eqn:c-monot-example}.
\end{proof}

\end{document}